\documentclass{amsart}
\usepackage{amssymb}
\usepackage{hyperref}
\newcommand{\la}{\lambda}

\newcommand{\om}{\omega}

\newcommand{\be}{\begin{equation}}
	\newcommand{\ee}{\end{equation}}
\newcommand{\bea}{\begin{eqnarray}}
	\newcommand{\eea}{\end{eqnarray}}
\newcommand{\nn}{\nonumber}
\newcommand{\bee}{\begin{eqnarray*}}
	\newcommand{\eee}{\end{eqnarray*}}
\newcommand{\ba}{\begin{aligned}}
	\newcommand{\ea}{\end{aligned}}
\newcommand{\bp}{\begin{proof}}
	\newcommand{\ep}{\end{proof}}  
\newcommand{\br}{\begin{remark}}
	\newcommand{\er}{\end{remark}}  
\newcommand{\lb}{\label}

\newtheorem{thm}{Theorem}[section]
\newtheorem{cor}[thm]{Corollary}
\newtheorem{lem}[thm]{Lemma}
\newtheorem{prop}[thm]{Proposition}
\theoremstyle{definition}

\theoremstyle{remark}

\usepackage{enumitem}

\newcommand{\Field}{\mathbb{F}}

\begin{document}
	\title[]{Arithmetic progression in a finite field with prescribed norms }
	\author[]{Kaustav Chatterjee, Hariom Sharma, Aastha Shukla, and Shailesh Kumar Tiwari$^*$}
	\thanks{*email: shaileshiitd84@gmail.com}
	\thanks{First author is supported by National Board for Higher Mathematics(IN), Ref No. 0203/6/2020-
		RD-11/7387.}
	\maketitle
	\begin{abstract}
		Given a prime power $q$ and a positive integer $n$, let $\Field_{q^{n}}$ represents a finite extension of degree $n$ of the finite field ${\mathbb{F}_{q}}$. In this article, we investigate the existence of $m$ elements in arithmetic progression, where every element is primitive and at least one is normal with prescribed norms. Moreover, for $n\geq6,q=3^k,m=2$ we establish that there are only $10$ possible exceptions.  
	\end{abstract}
	\section{Introduction}
	Let $\Field_{q}$ be a finite field with $q$ elements, where $q=p^k$ for some prime $p$ and positive integer $k$. With respect to multiplication, the set $\Field_{q}^{*}$, consisting of the non-zero elements in $\Field_{q}$, forms a cyclic group and a generator of this multiplicative group is defined to be a primitive element of the finite field $\Field_{q}$. The finite field $\Field_{q}$ contains $\phi(q-1)$ primitive elements, with $\phi$ representing Euler's totient function. An element $\alpha\in\Field_{q^{n}}^*$ is primitive if and only if is a root of an irreducible polynomial of degree $n$ over $\Field_{q}$ and such an irreducible polynomial is termed as primitive polynomial. Let $\Field_{p^n}$ represents a field extension of $\Field_{q}$ of degree of $n$, where $n$ is a positive integer. For any $\alpha$ $\in\Field_{q^n}$, the elements $\alpha, \alpha^{q}, \ldots, \alpha^{q^{n-1}}$ are referred as the conjugates of $\alpha$ with respect to $\Field_{q}$.  An element $\alpha\in\Field_{q^n}$ is said to be normal over $\Field_{q}$ if the set of conjugates of $\alpha$ forms a basis of $\Field_{q^n}$ over $\Field_q$, while the basis itself is called normal basis. The \textit{norm} of an element $\alpha \in \Field_{q^n}$ over $\Field_{q}$, denoted by $N_{{q^n}/q}(\alpha)$, is the product of the conjugates of $\alpha$ over $\Field_q$, that is, $N_{\Field_{q^{n}}/\Field_{q}}(\alpha)=\alpha\cdot\alpha^{q}\cdot\alpha^{q^2}\ldots\alpha^{q^{n-1}}$. We refer reader \cite{RH} for further informations related to the existence of primitive normal elements. 
	
	Primitive elements play a foundational role in Random Number Generation, Signal Processing and Elliptic Curve Cryptography.  A lot of applications of primitive elements can be found in Coding theory and Cryptography [\cite{PP}, \cite{AMOV}], making the study of primitive polynomials and primitive elements an important research area. For any $\alpha,\beta\in\Field_{q^n}$, a pair $(\alpha,\beta)$ is referred as primitive normal pair if both $\alpha$ and $\beta$ are primitive as well as normal over $\Field_{q}$. In particular, for any $f(x)\in\Field_{q^{n}}(x)$, the existence of primitive normal pairs $(\alpha,f(\alpha))$ has been a much interesting research area and many researchers has worked in the direction [\cite{RKS},\cite{RAA},\cite{AMRA},\cite{HHDK}].  
	
	In $1987$, Lenstra and Schoof \cite{HWRJ} proved the \textit{The Primitive Normal Basis Theorem}, which guarantees that every extension $\Field_{q^{n}}$ over $\Field_{q}$ always contains a primitive normal element.  In $2000$, Cohen and Hachenberger \cite{SH} demonstrated the existence of primitive and normal element $\alpha\in\Field_{q}$ with prescribed trace and norm. Recently, Sharma et al. \cite{AMS} proved that for any polynomial $f(x)\in\Field_{q^{n}}[x]$ with minor restrictions, there exists a primitive normal pair $(\alpha,f(\alpha))$ such that $Tr_{\Field_{q^n}/\Field_q}(\alpha)=b$ and $N_{\Field_{q^n}/\Field_q}(\alpha)=a$, for any prescribed $a,b\in\Field_{q}$. In \cite{AV}, Lemos et al. investigated the existence of an element $\alpha\in\Field_{q^{n}}^*$ such that for fixed $\beta\in\Field_{q^{n}}^*$, any element in the set $\{\alpha,\alpha+\beta,\ldots,\alpha+(m-1)\beta\}$ is primitive and at least one of them is normal. In this article, we establish the sufficient condition for the existence of an element $\alpha\in\Field_{q^{n}}$ such that for fixed $\beta\in\Field_{q^n}^*$, each element in the $m$-term arithmetic progression $\alpha,\alpha+\beta,\ldots,\alpha+(m-1)\beta$ is primitive, at least one of them is normal with $N_{\Field_{q^n}/\Field_q}(\alpha+(i-1)\beta)=c_{i}$, where $c_{i}\in\Field_{q}$ for $i\in\{1,2,\ldots,m\}$. Here we must have $m$ should be less than or equal to the characteristic of $\Field_{q}$.

	The structure of this article is as follows: In Section \ref{Sec2}, we give basic definitions that will be used  throughout the article. Section \ref{Sec3} includes some characteristic functions and lemmas providing bounds on characteristic sums.  In Section \ref{Sec4}, we figure out  the sufficient condition for the existence of the mentioned arithmetic progressions, along with an additional condition derived through sieve methods. Section \ref{Sec5} contains numerical examples for further illustration.
	\section{Preliminaries}\lb{Sec2}
	Before proceeding further, we provide some essential notations and significant definitions, which will be used consistently throughout the article. Assume that $n$ is a positive integer, $q$ is a prime power and $\Field_{q}$ is a finite field with $q$ elements. 
	\subsection{Definitions}
	\begin{itemize}
		\item[1.] Let $G$ be an abelain group and $T$ be the subset of complex numbers containing elements with unit modulus. A character $\chi$ of $G$ is a homomorphism from $G$ into $T$ i.e., $\chi(g_{1}g_{2})=\chi(g_{1})\chi(g_{2})$ for all $g_{1},g_{2}\in G$. The set of characters of ${G}$, denoted by $\widehat{{G}}$, forms a group under multiplication. In addition, $G$ is isomorphic to $\widehat{{G}}$. In the context of a finite field $\Field_{q^{n}}$, an additive character relates to the additive group of $\Field_{q^{n}}$, while a multiplicative character corresponds to the multiplicative group of $\Field_{q^{n}}^*$.\\
		\item[2.]	Let $e|q^n-1$ and $\alpha \in \Field_{q^{n}}^{*}$. Suppose for any $d|e$  and $\beta \in\Field_{q^{n}}$, $\alpha=\beta^{d}$ implies $d=1$. Then we say that $\alpha$ is a $e$-free element. Thus $\alpha\in\Field_{q^{n}}^*$ is primitive $\iff$ $\alpha$ is $(q^{n}-1)$-free.\\
		\item[3.] 	Let $\alpha\in\Field_{q^{n}}$ and $g$ be any divisor of $x^n-1$. We say that $g$ is $\Field_{q}$-order of $\alpha$ if $\alpha=h\circ\beta$, where $\beta\in\Field_{q^{n}}$ and $h|g$ implies $h=1$. It is easy to observe that an element $\alpha\in\Field_{q^{n}}$ is normal if and only if $\alpha$ is $x^n-1$-free.\\
		\item[4.]  The Euler's totient function for any polynomial $f(x)\in\Field_{q}[x]$ is defined as follows:
		\begin{equation}\nonumber
			\Phi_{q}(f)=\Bigg|\Bigg(\frac{\Field_{q}[x]}{f\Field_{q}[x]}\Bigg)^{*}\Bigg|=q^{deg(f)}\cdot \prod_{p\in\mathcal{I}_{f}}^{}\Bigg(1-\frac{1}{p}\Bigg)
		\end{equation}
		where $\mathcal{I}_{f}$ represents the set of irreducible factors of $f$ and $f\mathbb{F}_{q}[x]$ is the ideal generated by $f$.\\
		\item[5.] The M\"{o}bius function for the set of polynomials over $\Field_{q}$ is defined as follows: In case when $f$ is product of $r$ distinct monic irreducible polynomials then $\mu_{q}(f)=(-1)^{r}$ and otherwise $\mu_{q}(f)=0$.
	\end{itemize}
	\section{Some useful results}\lb{Sec3}
	In this section initially we present the characteristic functions of $e$-free, $g$-free elements and the set of elements in $\Field_{q^{n}}$ with prescribed norm $c$.
	
	For any $e|(q^n-1)$, following the work of  Cohen and Huczynska \textbf{{\cite{SS}}}, the characteristic function of the $e$-free elements of $\Field_{q^n}^*$ is given by
	$$\rho_{e}:\alpha \mapsto{}\frac{\phi(e)}{e}\sum_{k|e}\frac{\mu(k)}{\phi(k)}\sum_{\chi_{k}}\chi_{k}(\alpha),$$
	where the inner sum is taken over all multiplicative characters of order $k$, $\mu(.)$ is the M\"{o}bius function and $\phi(.)$ is the Euler function.
	
	Given $g(x)|x^{n}-1$, the characteristic function of the set of $g$-free elements in $\Field_{q^{n}}$ is given by
	$$\kappa_{g}:\alpha \mapsto{}\frac{\Phi_{q}(g)}{q^{deg(g)}}\sum_{h|g}\frac{\mu_{q}(h)}{\Phi_{q}(h)}\sum_{\la_{h}}\la_{h}(\alpha),$$
	where the inner sum is taken over any additive characters of $\Field_{q}$-order $h$, $\mu_{q}(.)$ is the M\"{o}bius function on $\Field_{q}[x]$ and $\Phi_{q}(.)$ is the Euler function on $\Field_{q}[x]$.
	
	Also, for any $c\in\Field_q^*$,
	\begin{equation}\nonumber
		\begin{aligned}
			\pi_{c}:\alpha \mapsto \frac{1}{q-1}\sum_{\chi\in{\widehat{{\Field}_{q}^*}}}^{}\chi(N_{\Field_{q^n}/\Field_q}(\alpha)c^{-1})
		\end{aligned}
	\end{equation}
	is a characteristic function for the subset of $\Field_{q^n}$ containing those elements $\alpha$ such that $N_{\Field_{q^n}/\Field_q}(\alpha)=c$. Let $\chi_{q-1}$ represents a generator of the multiplicative group of characters $\widehat{{\Field}_{q}^*}$. Then, any $\chi\in\widehat{{\Field}_{q}^*}$ can be expressed as $\chi(\alpha)=\chi_{q-1}(\alpha^{i})$ for some $i\in\{1,2,\ldots,q-1\}$. Thus we get
	\begin{equation}\nonumber
		\begin{aligned}
			\pi_{c}(\alpha)&=\frac{1}{q-1}\sum_{i=1}^{q-1}\chi_{q-1}^{i}(N_{\Field_{q^n}/\Field_q}(\alpha)c^{-1})\\
			&=\frac{1}{q-1}\sum_{i=1}^{q-1}\bar{\chi}^{i}(\alpha)\chi_{q-1}(c^{-i}),
		\end{aligned}
	\end{equation}
	where $\bar{\chi}=\chi_{q-1}\circ N_{\Field_{q^n}/\Field_q}$ is a multiplicative character of $\Field_{q^n}^*$. Following \cite{AMS}, the order of $\bar{\chi}$ in $\widehat{{\Field}_{q^n}^*}$ is $q-1$ and thus there exists a multiplicative character $\chi_{q^{n}-1}\in\widehat{{\Field}_{q^n}^*}$ of order $q^{n}-1$ such that $\bar{\chi}=\chi_{q^n-1}^{{q^{n}-1}/{q-1}}$.
	
	Following two lemmas provide bounds on the character sums which will be useful to prove our sufficient condition.
	\begin{lem}\textbf{({\cite{LDQ}}, Theorem 5.5)}\lb{L1}
		Let $f(x)=\prod_{j=1}^{s}f_{j}(x)^{a_{j}}$ be a rational function over $\Field_{q^{n}}$, where $f_{j}(x)\in\Field_{q^{n}}[x]$ are polynomials and $a_{j}\in\mathbb{Z}\smallsetminus\{0\}$. Let $\chi\in\widehat{{\Field}_{q}^*}$ be a multiplicative character of order $d$. Consider the rational function $f(x)$ that is not of the form of $r(x)^{d}$, where $r(x)$ belongs to the field $\mathbb{F}(x)$ and $\mathbb{F}$ represents the algebraic closure of $\Field_{q^{n}}$. Then we have $$\Bigg|\sum_{\alpha\in\Field_{q^n} }^{}\chi(f(\alpha))\Bigg|\leq\Bigg(\sum_{i=1}^{s}deg(f_{i})-1\Bigg)q^{n/2}.$$	
	\end{lem}
	\begin{lem}\textbf{({\cite{LDQ}}, Theorem 5.6)}\lb{L2}
		Let us Consider two rational functions, $f(x)$ and $g(x)$, both belonging to the field $\mathbb{F}_{q^n}(x)$. Express $f(x)$ as $\prod_{j=1}^{s} f_{j}(x)^{a_j}$, where each $f_j(x)$ are irreducible polynomials over the field $\mathbb{F}_{q^n}$ and $a_j\in\mathbb{Z}\smallsetminus \{0\}$. Let $D_{1}=\sum_{j=1}^{s}deg(f_{j})$, $D_{2}=max(deg(g(x)),0)$, $D_{3}$ is the degree of the denominator of $g(x)$ and the sum of the degrees of the irreducible polynomials dividing the denominator of $g(x)$ (excluding those that are equal to $f_{j}(x)$, for some $j=1,2,\ldots,s$) is equal to $D_{4}$. Let $\chi$ be a multiplicative character of $\Field_{q^n}^*$ and $\la$ is nontrivial additive character of $\Field_{q^{n}}$. Also, suppose that $g(x)$ is not of the form $ H(x)^{q^n}-H(x)$, $H(x)\in\mathbb{F}_{q^n}(x)$. Then
		\begin{equation} \nonumber
			\begin{aligned}
				\Bigg|\sum_{\alpha\in \Field_{q^{n}},f(\alpha)\neq 0,\infty, g(\alpha)\neq \infty }\chi(f(\alpha))\la(g(\alpha))\Bigg|\leq(D_{1}+D_{2}+D_{3}+D_{4}-1)q^{n/2}.
			\end{aligned}
		\end{equation}
	\end{lem}
	Let $r(>1)$ be a positive integer and $g\in\Field_{q}[x]$. We use $W(r)$ to represent the number of square-free divisors of $r$ and $W(g)$ indicates the number of square-free factors of $g$.
	\begin{lem}\textbf{({\cite{SS}}, Lemma 3.7)}\lb{L3}
		Let $\nu$ be a positive real number and $r\in\mathbb{N}$. Then we have $W(r)<\mathcal{C}\cdot r^{\frac{1}{\nu}}$, where $\mathcal{C}$ is defined as $\frac{2^u}{{({p_{1}p_{2}\ldots p_{u})}}^{\frac{1}{\nu}}}$ and $p_{1},p_{2},\ldots,p_{u}$ are primes $\leq 2^{\nu}$ that divide $r$.\\
	\end{lem}
	\begin{lem}\textbf{({\cite{HWRJ}}, Lemma 2.9)}\lb{L4}
		Let $q$ be a prime power and $n$ be a positive integer. Then we have $W(x^n-1)\leq 2^{\frac{1}{2}\{n+gcd(n,q-1)\}}$. Moreover, $W(x^n-1)\leq 2^n$ and $W(x^n-1)=2^n$ if and only if $n$ divides $q-1$. Additionally, we have $W(x^n-1)\leq 2^{\frac{3}{4}n}$ if $n\nmid q-1$.
	\end{lem}
	
	We know that norm of a primitive element is also primitive. Moreover, Sharma et al. {{\cite{AMS}}} has proved the following lemma in a more general context.
	\begin{lem}\textbf{({\cite{AMS}}, Lemma 3.1)}\lb{L5}
		Let $e$ be a positive divisor of $q^n-1$ and $\delta=$gcd($e,q-1$). Additionally, let $L_{e}$ represents the largest divisor of $e$ for which gcd($L_{e},\delta$)=1. Then an element $\alpha\in\Field_{q^n}^*$ is $l$-free $\iff$ $\alpha$ is $L_{e}$-free and $N_{\Field_{q^n}/\Field_q}(\alpha)$ is $\delta$-free. 
	\end{lem}
	Since $L_{q^n-1}$ is the largest divisor of $q^n-1$ with $gcd(L_{q^n-1},q-1)=1$, we can express $L_{q^n-1}$ as $\frac{q^{n}-1}{(q-1)gcd(n,q-1)}$ and this further implies $W(L_{q^n-1})\leq W(q^{n}-1)$.
	\section{Main Results}\lb{Sec4}
	Let $\beta\in\Field_{q^{n}}^*$ and $c_{1},c_{2},\ldots,c_{m}\in\Field_{q}^*$. We shall establish a sufficient condition which ensures the existence of an element $\alpha\in\Field_{q^{n}}^*$ such that the elements $\alpha,\alpha+\beta,\ldots,\alpha+(m-1)\beta$ are all primitive and at least one of them is normal together with $N_{\Field_{q^{n}}/\Field_{q}}(\alpha)=c_{1}$, $N_{\Field_{q^{n}}/\Field_{q}}(\alpha+\beta)=c_{2}$,\ldots, $N_{\Field_{q^{n}}/\Field_{q}}(\alpha+(m-1)\beta)=c_{m}$. Let $e_{1},e_{2},\ldots,e_{m}$ be positive divisors of $q^{n}-1$, $g\in\Field_{q}[x]$ be a  divisor of $x^{n}-1$ and $\delta_{i}=$gcd($e_{i},q-1)$ for $i=1,2,\ldots,m$. Further, let ${L}_{e_{i}}$ be the largest divisor of $e_{i}$ such that $gcd(L_{e_{i}},\delta_{i})=1$ for $i=1,2,\ldots,m$. To proceed further, let us consider the following definitions.
	\subsection{Definitions.}
	\begin{itemize}
		\item[1.] For $m\geq 3$, let $\mathcal{S}_{m}$ be the set containing the pairs $(q,n)\in\mathbb{N}\times\mathbb{N}$ such that for any $\beta\in\Field_{q^{n}}^*$ and a subset $\{c_{i}\in\Field_{q}:1\leq i \leq m\}\subset\Field_{q}^*$, there exists an element $\alpha\in\Field_{q^{n}}^*$ such that the elements in the set $\{\alpha+(i-1)\beta:1\leq i \leq m\}$ are all primitive, at least one is normal and $N_{\Field_{q^n}/\Field_q}(\alpha+(i-1)\beta)=c_{i}$ for $i=1,2,\dots,m$. \\
		\item[2.] For $m\in\mathbb{N}$, we define $\mathcal{C}_{m}$ to be the set containing the $m$-tuples $(c_{1},c_{2},\ldots,c_{m}),$ where $c_{i}\in\Field_{q}^*$ be such that $c_{i}$ is $\delta_{i}$-free for $i=1,2,\dots,m$.\\
		\item[3.] For any $(c_{1},c_{2},\ldots,c_{m})\in\mathcal{C}_{m}$, $\Psi_{c_{1},c_{2},\ldots,c_{m}}({e_{1}},{e_{2}},\ldots,{e_{m}},g)$ denotes the cardinality of the set containing the elements $\alpha\in\Field_{q^{n}}^*$ that satisfy the following conditions:
		\begin{enumerate}
			\item[a.] $\alpha+(i-1)\beta$ is ${e_{i}}$-free for $i=1,2,\ldots,m$,
			\item[b.] $\alpha+(j-1)\beta$ is $g$-free for some $j\in\{1,2,\ldots,m\}$,
			\item[c.] $N_{\Field_{q^n}/\Field_q}(\alpha+(i-1)\beta)=c_{i}$.
		\end{enumerate}
		\item[4.] For any $(c_{1},c_{2},\ldots,c_{m})\in\mathcal{C}_{m}$ and $j\in\{1,2,\ldots,m\}$, we define the set $\Psi_{c_{1},c_{2},\ldots,c_{m}}({e_{1}},{e_{2}},\ldots,{e_{m}},g,j)$ as the number of the elements $\alpha\in\Field_{q^{n}}^*$ that satisfy the following conditions:
		\begin{enumerate}
			\item[a.] $\alpha+(i-1)\beta$ is ${e_{i}}$-free for $i=1,2,\ldots,m$,
			\item[b.] $\alpha+(j-1)\beta$ is $g$-free,
			\item[c.] $N_{\Field_{q^n}/\Field_q}(\alpha+(i-1)\beta)=c_{i}$.
		\end{enumerate}
	\end{itemize}
	We shall use the following abbreviations consistently throughout this article:
	\begin{equation}\nonumber
		\begin{aligned}
			\bar{c}=&(c_{1},c_{2},\ldots,c_{m}),~ \overline{L_{{e}}}=(L_{e_{1}},L_{e_{2}},\ldots,L_{e_{m}}), W(\overline{L_{{e}}})=\prod_{i=1}^{m}W({L_{{e_{i}}}})\\
			&\text{and}~ \bar{e}=\overline{q^n-1} ~\text{when }~ e_{i}=q^n-1 ~\text{for}~ i=1,2,\ldots,m.
		\end{aligned}
	\end{equation}
	Thus we have
	\begin{equation}\lb{eq1}
		\begin{aligned}
			\Psi_{\bar{c}}(\overline{L_{{e}}},g)\geq \frac{1}{m}\sum_{j=1}^{m}\Psi_{\bar{c}}(\overline{L_{{e}}},g,j).
		\end{aligned}
	\end{equation}
	We now prove our principal result, as given in the following.
	\begin{thm}
		Let $q,n,m$ be positive integers such that $q$ is a prime power and $n>2m$. Suppose that
		\begin{equation}\nonumber
			\begin{aligned}
				q^{\frac{n}{2}-m}\geq mW(\overline{L_{{q^{n}-1}}})W(x^{n}-1).
			\end{aligned}
		\end{equation}
		Then $(q,n)\in \mathcal{S}_{m}$.
	\end{thm}
	\begin{proof}
		Let $e_{1},e_{2},\ldots,e_{m}$ be positive divisors of $q^{n}-1$, $g\in\Field_{q}[x]$ be a divisor of $x^{n}-1$ and $\beta\in\Field_{q^{n}}^{*}$. Let $Z=\{-(i-1)\beta:1\leq i\leq m\}$. From \ref{eq1}, using the definitions of the characteristic functions $\rho_{e}$, $\kappa_{g}$ and $\pi_{c}$, we have 
		\begin{equation}\lb{eq2}
			\begin{aligned}
				\Psi_{\bar{c}}(\overline{L_{{e}}},g)&\geq \frac{1}{m}\sum_{\alpha\in\Field_{q^{n}}\smallsetminus Z}^{}\bigg[\prod_{i=1}^{m}\rho_{L_{e_{i}}}(\alpha+(i-1)\beta)\pi_{c_{i}}(\alpha+(i-1)\beta)\sum_{j=1}^{m}\kappa_{g}(\alpha+(j-1)\beta)\bigg]\\& 
				=\frac{{\mathfrak{C}}}{m}\underset{1\leq i\leq m}{\sum_{d_{i}|L_{e_{i}}}^{}}\underset{}{\underset{h|g}{\sum}}\frac{\mu}{\phi}(d_{1},d_{2},\ldots,d_{m},h) \underset{1\leq i\leq m}{\sum_{\chi_{d_{i}}}^{}}\sum_{{\la_{h}}}^{}\boldsymbol{\chi}_{\bar{d},h,\bar{c}},
			\end{aligned}
		\end{equation}
		where  $\mathfrak{C}=\frac{\phi(L_{e_{1}})\phi(L_{e_{2}})\ldots\phi(L_{e_{m}})\Phi_{q}(g)}{L_{e_{1}}L_{e_{2}}\ldots L_{e_{m}}q^{deg(g)}(q-1)^m}$, $\frac{\mu}{\phi}(d_{1},d_{2},\ldots,d_{m},h)=\frac{\mu(d_{1})\mu(d_{2})\ldots\mu(d_{m})\mu_{q}(h)}{\phi(d_{1})\phi(d_{2})\ldots\phi(d_{m})\Phi_{q}(h)}$ 
		and 
		\begin{equation}\nonumber
			\begin{aligned}
				\boldsymbol{\chi}_{\bar{d},h,\bar{c}}=\underset{\underset{}{}}{\sum_{k_{1},k_{2},\ldots,k_{m}=1}^{q-1}}\chi_{q-1}(c_{1}^{-k_{1}}c_{2}^{-k_{2}}\ldots &c_{m}^{-k_{m}})\sum_{\alpha\in\Field_{q^{n}}\smallsetminus Z}^{}\prod_{i=1}^{m}\chi_{d_{i}}(\alpha+(i-1)\beta)\\&\times{\bar{\chi}}^{k_{i}}(\alpha+(i-1)\beta)\sum_{j=1}^{m}\la_{h}(\alpha+(j-1)\beta).
			\end{aligned}
		\end{equation}
		Consider $\chi_{q^{n}-1}$ a generator of the multiplicative cyclic group $\Field_{q^{n}}^*$. Thus, for $i\in\{1,2,\ldots,m\}$ there exist $t_{i}\in\{0,1,2,\ldots,q^n-2\}$ such that $\chi_{d_{i}}(\alpha)=\chi_{q^n-1}^{t_{i}}(\alpha)$. Further, there exists $y\in\Field_{q^{n}}$ such that $\la_{h}(\alpha)=\widehat{\la_{0}}(y\alpha)$. Here $\widehat{\la}_{0}$ is the additive character of $\Field_{q^n}$ defined by $\widehat{\la}_{0}(\alpha)=\la_{0}(Tr_{\Field_{q^n}/\Field_q}(\alpha))$, where $\la_{0}$ is the canonical additive character of $\Field_{q}$. Hence
		\begin{equation}\lb{eq3}
			\begin{aligned}
				\boldsymbol{\chi}_{\tilde{d},h,\tilde{c}}=\underset{\underset{}{}}{\sum_{k_{1},k_{2},\ldots,k_{m}=1}^{q-1}} \chi_{q-1}(c_{1}^{-k_{1}}c_{2}^{-k_{2}}\ldots &c_{m}^{-k_{m}})\sum_{j=1}^{m}\widehat{\la_{0}}((j-1)y\beta)\\&\times\sum_{\alpha\in\Field_{q^{n}}\smallsetminus {Z}}^{}\chi_{q^{n}-1}(F(\alpha))\widehat{\la_{0}}(y\alpha).
			\end{aligned}
		\end{equation}
		where $F(x)=\prod_{i=1}^{m}\{x+(i-1)\beta\}^{{\frac{q^{n}-1}{q-1}k_{i}+t_{i}}}$. If we have $F(x)=H(x)^{q^n-1}$ for some $H(x)\in\mathbb{F}_{q^n}(x)$, then comparing the powers of $x+(i-1)\beta$ we get $\frac{q^{n}-1}{q-1}k_{i}+t_{i}=(q^n-1)r_{i}$, for some non negative integer $r_{i}$. From this equation it follows that $t_{i}=(q^n-1)(r_{i}-\frac{k_{i}}{q-1})\geq (q^n-1)(r_{i}-1)$. Since $t_{i}\in\{0,1,\ldots,q^n-2\}$, the latter one implies that $r_{i}=1$. Thus we have $t_{i}=(q^n-1)(1-\frac{k_{i}}{q-1})$ for $i=1,2,\ldots,m$. As $\chi_{d_{i}}=\chi_{q^{n}-1}^{t_{i}}$, there is a positive integer $k_{i}^{'}$ such that $t_{i}d_{i}=(q^n-1)k_{i}^{'}$, which further implies $k_{i}^{'}(q-1)=d_{i}(q-1-k_{i})$. Since $gcd(q-1,d_{i})=1$, it follows that $d_{i}|k_{i}^{'}$. Thus we have $q^n-1|t_{i}$, which is possible only if $t_{i}=0$. Hence $\bar{t}=\bar{0}$, that is,  $\bar{d}=\bar{1}$.
		
		Moreover, we have $yx\neq H(x)^{q^n}-H(x)$ for any $H(x)\in\mathbb{F}_{q^n}(x)$ unless we have $H(x)$ is constant and $y=0$, that is $h=1$. Consequently, if $\bar{d}\neq \bar{1}$ or $h\neq 1$ then from the Lemma \ref{L2} we have
		\begin{equation}\lb{eq4}
			\begin{aligned}
				|\boldsymbol{\chi}_{\tilde{d},h,\tilde{c}}|&\leq m^{2}q^{\frac{n}{2}+m}.
			\end{aligned}
		\end{equation} 
		From (\ref{eq3}), we have
		\begin{equation}\lb{eq5}
			\begin{aligned}
				\boldsymbol{\chi}_{\tilde{1},1,\tilde{c}}=m\underset{\underset{}{}}{\sum_{k_{1},k_{2},\ldots,k_{m}=1}^{q-1}}\chi_{q-1}(c_{1}^{-k_{1}}c_{2}^{-k_{2}}\ldots c_{m}^{-k_{m}})\sum_{\alpha\in\Field_{q^{n}}\smallsetminus {Z}}^{}\chi_{q^{n}-1}(\bar{F}(\alpha)),
			\end{aligned}
		\end{equation}
		where $\bar{F}(x)=\prod_{i=1}^{m}\{x+(i-1)\beta\}^{{\frac{q^{n}-1}{q-1}k_{i}}}$. Suppose that $F_{i}(x)=x+(i-1)\beta$ for $i=1,2,\ldots,m$. We consider the following cases based on possible values of $d_{i}$'s.\\
		
		$\textbf{Case(1)}$: Let $S_{1}$ be the case when $k_{i}=q-1$ for $i=1,2,\ldots,m$. Clearly, we have $S_{1}=m(q^{n}-m)$.
		
		$\textbf{Case(2)}$: Let $S_{2}$ be the case when for fixed $i\in\{1,2,\ldots,m\}$, $k_{i}$ takes value in $\{1,2,\ldots,q-2\}$ and rest others are equal to $(q-1)$. There are total `$m$' such possibilities. Thus 
		\begin{equation}\nonumber
			\begin{aligned}
				|S_{2}|&\leq m^{2}\Bigg|\sum_{k_{i}=1}^{q-2}\chi_{q-1}(c_{i}^{-k_{i}})\sum_{\alpha\in\Field_{q^{n}}\smallsetminus Z}^{}\chi_{q^{n}-1}^{{\frac{q^{n}-1}{q-1}k_{i}}}(F_{i}(\alpha))\Bigg|.
			\end{aligned}
		\end{equation}
		Since $\chi_{q^{n}-1}^{{\frac{q^{n}-1}{q-1}k_{i}}}$ is a nontrivial multiplicative character in $\widehat{\Field_{q^{n}}^*}$, we must have 
		\begin{equation}\nonumber
			\begin{aligned}
				\sum_{\alpha\in\Field_{q^{n}}\smallsetminus \{-(i-1)\beta\}}^{}\chi_{q^{n}-1}^{{\frac{q^{n}-1}{q-1}k_{i}}}(F_{i}(\alpha))=0,
			\end{aligned}
		\end{equation}
		which gives
		\begin{equation}\nonumber
			|S_{2}|\leq m^{2}(m-1)(q-2).
		\end{equation}
		
		$\textbf{Case(3)}$: Let $S_{3}$ be the case when for a given pair $(i,j)$ of distinct integers belonging to the set $\{1,2,\ldots,m\}$, the variables $k_i$ and $k_j$ take values in the range ${1,2,\ldots,q-2}$, while all other $k_{p}$'s are set to $(q-1)$. There are a total of `$\binom{m}{2}$' possibilities in this case. Hence we have
		\begin{equation}\nonumber
			\begin{aligned}
				|S_{3}|&\leq m\binom{m}{2}\Bigg|\sum_{k_{i},k_{j}=1}^{q-2}\chi_{q-1}(c_{i}^{-k_{i}}c_{j}^{-k_{j}})\sum_{\alpha\in\Field_{q^{n}}\smallsetminus {Z}}^{}\chi_{q^{n}-1}\Bigg(F_{i}(\alpha)^{{\frac{q^{n}-1}{q-1}k_{i}}}F_{j}(\alpha)^{{\frac{q^{n}-1}{q-1}k_{j}}}\Bigg)\Bigg|.
			\end{aligned}
		\end{equation}
		Clearly we have $F_{i}(x)^{{\frac{q^{n}-1}{q-1}k_{i}}}F_{j}(x)^{{\frac{q^{n}-1}{q-1}k_{j}}}\neq H(x)^{q^{n-1}}$ for any $H(x)\in\mathbb{F}_{q^n}(x)$. Thus by Lemma \ref{L1}, we get
		$$|S_{3}|\leq m\binom{m}{2}(q-2)^2\{q^{\frac{n}{2}}+(m-2)\}.$$
		
		$\textbf{Case(4)}$: Let $S_{4}$ be the case when for a given triple $(i,j,l)$ of distinct integers belonging to the set $\{1,2,\ldots,m\}$, the variables $k_i$, $k_j$ and $k_{l}$ each taking values in the range ${1,2,\ldots,q-2}$, while all other $k_{p}$'s are set to $(q-1)$. There are a total of `$\binom{m}{3}$' possibilities in this case. Proceeding in a similar manner as done in the previous case, we have
		$$|S_{4}|\leq m\binom{m}{3}(q-2)^3\{2q^{\frac{n}{2}}+(m-3)\}.$$
		$$\vdots$$\\
		
		$\textbf{Case(m+1)}$: Let $S_{m+1}$ be the case when for each $i\in\{1,2,\ldots,m\}$, $k_{i}$ take values in the range $1,2,\ldots,q-2$. Here we have
		\begin{equation}\nonumber
			\begin{aligned}
				S_{m+1}=m\underset{\underset{}{}}{\sum_{k_{1},k_{2},\ldots,k_{m}=1}^{q-2}}\chi_{q-1}(c_{1}^{-k_{1}}c_{2}^{-k_{2}}\ldots c_{m}^{-k_{m}})\sum_{\alpha\in\Field_{q^{n}}\smallsetminus {Z}}^{}\chi_{q^{n}-1}(\bar{F}(\alpha)),
			\end{aligned}	
		\end{equation}
		which implies 
		\begin{equation}\nonumber
			|S_{m+1}|\leq m(q-2)^m(m-1)q^{\frac{n}{2}}.
		\end{equation}
		Taking into the consideration of all the above possibilities and using the equation (\ref{eq5}),
		we get
		\begin{equation}\nonumber
			\begin{aligned}
				|\boldsymbol{\chi}_{\tilde{1},1,\tilde{c}}-m(q^n-m)|&\leq m\Bigg[m (q-2)(m-1)+\binom{m}{2}(q-2)^2\{q^{\frac{n}{2}}+(m-2)\}\\
				&+\binom{m}{3}(q-2)^3\{2q^{\frac{n}{2}}+(m-3)\}+\ldots+(q-2)^m(m-1)q^{\frac{n}{2}}\Bigg]\\&=m\sum_{i=1}^{m}\binom{m}{i}(q-2)^{i}\{(m-i)+q^{\frac{n}{2}}(i-1)\}.
			\end{aligned}
		\end{equation}
		If $m$ is a positive integer, then we have
		\begin{equation}\lb{eq6}
			(x+1)^{m}-1=\sum_{i=1}^{m}\binom{m}{i}x^{i}.
		\end{equation}
		Division by `$x$' in both sides, followed by differentiation, the latter equation transforms into
		\begin{equation}\lb{eq7}
			\begin{aligned}
				\sum_{i=1}^{m}\binom{m}{i}(i-1)x^{i}=mx(x+1)^{m-1}-(x+1)^{m}+1.
			\end{aligned}
		\end{equation}
		We rewrite (\ref{eq6}) as,
		\begin{equation}\nn
			\begin{aligned}
				\sum_{i=1}^{m-1}\binom{m}{i+1}x^{i+1}=(x+1)^{m}-mx-1.
			\end{aligned}
		\end{equation}
		Taking the derivative with respect to `$x$' on both sides of the above equation we get,
		\begin{equation}\nn
			\begin{aligned}
				\sum_{i=1}^{m-1}\binom{m}{i+1}(i+1)x^{i}=m(x+1)^{m-1}-m.
			\end{aligned}
		\end{equation}
		Let us apply the binomial identity $\binom{m}{i}(m-i)=\binom{m}{i+1}(i+1)$ in the previous equation. Thus
		\begin{equation}\lb{eq8}
			\sum_{i=1}^{m}\binom{m}{i}(m-i)x^{i}=m(x+1)^{m-1}-m.
		\end{equation}
		Substituting `$x$' with `$(q-2)$' in (\ref{eq7}) and (\ref{eq8}), we obtain the following summations 
		\begin{equation}\nonumber
			\sum_{i=1}^{m}\binom{m}{i}(i-1)(q-2)^{i}=m(q-2)(q-1)^{m-1}-(q-1)^{m}+1
		\end{equation}
		$$\text{and}$$
		\begin{equation}\nonumber
			\sum_{i=1}^{m}\binom{m}{i}(m-i)(q-2)^{i}=m(q-1)^{m-1}-m.
		\end{equation}
		Hence we have
		\begin{equation}\nonumber
			\begin{aligned}
				\boldsymbol{\chi}_{\tilde{1},1,\tilde{c}}&\geq m[(q^n-m)-q^{\frac{n}{2}}\{m(q-2)(q-1)^{m-1}-(q-1)^{m}+1\}
				-m\{(q-1)^{m-1}-1\}]\\
				&=m[q^n-q^{\frac{n}{2}}\{m(q-2)(q-1)^{m-1}-(q-1)^{m}+1\}-m(q-1)^{m-1}].
			\end{aligned}
		\end{equation}
		Using the above lower bound, it can be deduced from equations (\ref{eq3}) and (\ref{eq4}) that
		\begin{equation}
			\begin{aligned}\nonumber
				\Psi_{\bar{c}}(\overline{L_{{e}}},g)&\geq \mathcal{C}[q^n-q^{\frac{n}{2}}\{m(q-2)(q-1)^{m-1}-(q-1)^{m}+1\}
				-m(q-1)^{m-1}\\&-mq^{\frac{n}{2}+m}(W(\overline{L_{{e}}})W(g)-1)]\\
				&>\mathcal{C}\{q^n-mq^{\frac{n}{2}+m}W(\overline{L_{{e}}})W(g)+q^{\frac{n}{2}}(q-1)^m-m(q-1)^{m-1}\},
			\end{aligned}
		\end{equation}
		which is positive if we have $q^{\frac{n}{2}-m}>mW(\overline{L_{{e}}})W(g)$. Clearly, the result follows by letting $e=e_{1}=e_{2}=\ldots=e_{m}=q^n-1$ and $g=x^n-1$. 
	\end{proof}
	In the following lemma, we refer the sieving inequality, originally introduced by Kapetanakis in \cite{GK} and subsequently apply the modified form of this inequality.
	\begin{lem}\lb{L4.2}
		Let $q,n,m$ be positive integers such that $q$ is a prime power and $n>2m$. Let $e=e_{1}=e_{2}=\ldots=e_{m}$ be a divisor of $L_{q^n-1}$ and $p_{1},p_{2},\ldots,p_{r}$ be the distinct primes dividing $L_{q^n-1}$ but not $e$. Moreover, let $g$ be a divisor of $x^n-1$ and $g_{1},g_{2},\ldots,g_{s}$ be the distinct irreducible polynomials dividing $x^n-1$ but not $g$. Then for $j\in\{1,2,\ldots,m\}$, we have
		\begin{equation}\lb{eq9}
			\begin{aligned}
				\Psi_{\bar{c}}(\overline{L_{{q^n-1}}},x^n-1,j)&\geq \sum_{i=1}^{r}\Psi_{\bar{c}}(p_{i}e_{1},e_{2},\ldots,e_{m},g,j)+\sum_{i=1}^{r}		\Psi_{\bar{c}}(e_{1},p_{i}e_{2},\ldots,e_{m},g,j)\\
				&+\ldots+\sum_{i=1}^{r}\Psi_{\bar{c}}(e_{1},e_{2},\ldots,p_{i}e_{m},g,j)+\sum_{i=1}^{s}\Psi_{\bar{c}}(\bar{e},g_{i}g,j)\\&-(mr+s-1)\Psi_{\bar{c}}(\bar{e},g,j).\\
			\end{aligned}
		\end{equation}
	\end{lem} 
	\begin{lem}\lb{L4.3}
		Let $q,n,m$ be positive integers such that $q$ is a prime power and $n>2m$. Let $e=e_{1}=e_{2}=\ldots=e_{m}$ be a divisor of $L_{q^n-1}$ and $p$ be a prime dividing $L_{q^n-1}$ but not $e$. Moreover, let $g$ be a divisor of $x^n-1$ and $g'$ be an irreducible factors of $x^n-1$ but not $g$. Denote $\prod_{i=1}^{m}W(e_{i})$ and $\prod_{i=1}^{m}\theta(e_{i})$ by $W(\bar{e})$ and $\theta(\bar{e})$ respectively. Then for $j\in\{1,2,\ldots,m\}$, we have\\
		\begin{equation}\nonumber
			\begin{aligned}
				|\Psi_{\bar{c}}(p_{i}e_{1},e_{2},\ldots,e_{m},g,j)-\theta(p)\Psi_{\bar{c}}(\bar{e},g,j)|&\leq m^{2}q^{\frac{n}{2}+m}{\theta(\bar{e})\Theta(g)\theta(p)}W(\bar{e})W(g),\\
				|\Psi_{\bar{c}}(e_{1},pe_{2},\ldots,e_{m},g,j)-\theta(p)\Psi_{\bar{c}}(\bar{e},g,j)|&\leq m^{2}q^{\frac{n}{2}+m}{\theta(\bar{e})\Theta(g)\theta(p)}W(\bar{e})W(g),\\
				&\vdots\\ 
				|\Psi_{\bar{c}}(e_{1},e_{2},\ldots,pe_{m},g,j)-\theta(p)\Psi_{\bar{c}}(\bar{e},g,j)|&\leq m^{2}q^{\frac{n}{2}+m}{\theta(\bar{e})\Theta(g)\theta(p)}W(\bar{e})W(g),\\
				&\text{and}\\
				|\Psi_{\bar{c}}(e_{1},e_{2},\ldots,e_{m},gg',j)-\Theta(g')\Psi_{\bar{c}}(\bar{e},g,j)|&\leq m^{2}q^{\frac{n}{2}+m}{\theta(\bar{e})\Theta(g)\Theta(g')}W(\bar{e})W(g).
			\end{aligned}
		\end{equation}
	\end{lem}
	\begin{proof}
		From the definition, we have
		\begin{equation}\nonumber
			\Psi_{\bar{c}}(pe_{1},e_{2},\ldots,e_{m},g,j)-\theta(p)\Psi_{\bar{c}}(\bar{e},g,j)=\theta(\bar{e})\Theta(g)\theta(p)\underset{t\in\{2,3,\ldots,m\}}{\underset{d_{t}|e_{t}}{\underset{p|{d_{1}}|pe_{1},h|g}{\sum}}}\frac{\mu}{\phi}(\bar{d},h) \underset{1\leq i\leq m}{\sum_{\chi_{d_{i}},\la_{h}}^{}}\boldsymbol{\chi}_{\bar{d},h,\bar{c}}
		\end{equation}
		Now, using the inequality $|\boldsymbol{\chi}_{\bar{d},h,\bar{c}}|\leq m^{2}q^{\frac{n}{2}+m}$, we get
		\begin{equation}\nonumber
			|\Psi_{\bar{c}}(pe_{1},e_{2},\ldots,e_{m},g,j)-\theta(p)\Psi_{\bar{c}}(\bar{e},g,j)|\leq m^{2}q^{\frac{n}{2}+m}{\theta(\bar{e})\Theta(g)\theta(p)}W(\bar{e})W(g).
		\end{equation}
		Similarly we have,
		\begin{equation}\nonumber
			\begin{aligned}
				|\Psi_{\bar{c}}(e_{1},pe_{2},\ldots,e_{m},g,j)-\theta(p)\Psi_{\bar{c}}(\bar{e},g,j)|&\leq m^{2}q^{\frac{n}{2}+m}{\theta(\bar{e})\Theta(g)\theta(p)}W(\bar{e})W(g),\\
				|\Psi_{\bar{c}}(e_{1},pe_{2},\ldots,e_{m},g,j)-\theta(p)\Psi_{\bar{c}}(\bar{e},g,j)|&\leq m^{2}q^{\frac{n}{2}+m}{\theta(\bar{e})\Theta(g)\theta(p)}W(\bar{e})W(g),\\
				&\vdots \\
				|\Psi_{\bar{c}}(e_{1},e_{2},\ldots,pe_{m},g,j)-\theta(p)\Psi_{\bar{c}}(\bar{e},g,j)|&\leq m^{2}q^{\frac{n}{2}+m}{\theta(\bar{e})\Theta(g)\theta(p)}W(\bar{e})W(g).\\
			\end{aligned}
		\end{equation}
		Further, 
		\begin{equation}\nonumber		\Psi_{\bar{c}}(e_{1},e_{2},\ldots,e_{m},gg',j)-\Theta(g')\Psi_{\bar{c}}(\bar{e},g,j)=\theta(\bar{e})\Theta(g)\Theta(g')\underset{}{\underset{g'|h}{\underset{\bar{d}|\bar{e},h|gg'}{\sum}}}\frac{\mu}{\phi}(\bar{d},h) \underset{1\leq i\leq m}{\sum_{\chi_{d_{i}},\la_{h}}^{}}\boldsymbol{\chi}_{\bar{d},h,\bar{c}}.
		\end{equation}
		Hence 
		\begin{equation}\nonumber
			|\Psi_{\bar{c}}(e_{1},e_{2},\ldots,e_{m},gg',j)-\Theta(g')\Psi_{\bar{c}}(\bar{e},g,j)|\leq m^{2}q^{\frac{n}{2}+m}{\theta(\bar{e})\Theta(g)\Theta(g')}W(\bar{e})W(g).
		\end{equation}
	\end{proof}
	\begin{prop}\lb{Prop4.4}
		Assuming all the notations and conditions in the Lemma \ref{L4.2}, we define
		\begin{equation}\nonumber
			\begin{aligned}
				\delta:=1-m\sum_{i=1}^{r}\frac{1}{p_{i}}-\sum_{i=1}^{s}\frac{1}{q^{deg(g_{j})}}
			\end{aligned}
		\end{equation}  and 
		\begin{equation}\nonumber
			\begin{aligned}
				\Delta:=\frac{mr+s-1}{\delta}+2. \end{aligned}
		\end{equation}
		If $\delta>0$ and $$q^{\frac{n}{2}-m}>m{W(\bar{e})}W(g)\Delta$$ then $(q,n)\in\mathcal{S}_{m}$.
	\end{prop}
	\begin{proof}
		For fix $j\in\{1,2,\ldots,m\}$, let us rewrite the equation (\ref{eq9}) as
		\begin{equation}\nonumber
			\begin{aligned}
				\Psi_{\bar{c}}(\overline{L_{{q^n-1}}},x^n-1,j)&\geq \sum_{i=1}^{r}[\Psi_{\bar{c}}(p_{i}e_{1},e_{2},\ldots,e_{m},g,j)-\theta(p_{i})\Psi_{\bar{c}}(\bar{e},g,j)]\\&+\sum_{i=1}^{r}		[\Psi_{\bar{c}}(e_{1},p_{i}e_{2},\ldots,e_{m},g,j)-\theta(p_{i})\Psi_{\bar{c}}(\bar{e},g,j)]+\ldots\\
				&+\sum_{i=1}^{r}[\Psi_{\bar{c}}(e_{1},e_{2},\ldots,p_{i}e_{m},g,j)-\theta(p_{i})\Psi_{\bar{c}}(\bar{e},g,j)]\\&+\sum_{i=1}^{s}[\Psi_{\bar{c}}(e_{1},e_{2},\ldots,p_{i}e_{m},gg_{i},j)-\Theta(g_{i})\Psi_{\bar{c}}(\bar{e},g,j)]+\delta \Psi_{\bar{c}}(\bar{e},g,j).\\
			\end{aligned}
		\end{equation}
		We obtain the subsequent expression by applying the Lemma \ref{L4.3}, as
		\begin{equation}\nonumber
			\begin{aligned}
				\Psi_{\bar{c}}(\overline{L_{{q^n-1}}},x^n-1,j)\geq \frac{\phi(e)\cdot\Phi_{q}(g)}{e\cdot q^{deg(g)}}\Bigg[\Bigg(m&\sum_{i=1}^{r}\theta(p_{i})+\sum_{i=1}^{s}\Theta(g_{i})\Bigg)(-m^{2}W(\bar{e})W(g)q^{\frac{n}{2}+m})\\
				&+\delta\{m(q^n-m)-m^{2}q^{\frac{n}{2}+m}(W(\bar{e})W(g)-1)\}\Bigg],
			\end{aligned}
		\end{equation}
		which implies
		\begin{equation}\nonumber
			\begin{aligned}
				\Psi_{\bar{c}}(\overline{L_{{q^n-1}}},x^n-1,j)\geq \frac{\phi(e)\cdot\Phi_{q}(g)}{e\cdot q^{deg(g)}}&\delta\Bigg[\Bigg(\frac{m\sum_{i=1}^{r}\theta(p_{i})+\sum_{i=1}^{s}\Theta(g_{i})}{\delta}+1\Bigg)\times\\&(-m^{2}W(\bar{e})W(g) q^{\frac{n}{2}+m})
				+\{m(q^n-m)+m^{2}q^{\frac{n}{2}+m}\}\Bigg].
			\end{aligned}
		\end{equation}
		Here we note that $\delta=m\sum_{i=1}^{r}\theta(p_{i})+\sum_{i=1}^{s}\Theta(g_{i})-(mr+s-1)$. Then the above inequality becomes
		\begin{equation}\nonumber
			\begin{aligned}
				\Psi_{\bar{c}}(\overline{L_{{q^n-1}}},x^n-1,j)&\geq \frac{\phi(e)\cdot\Phi_{q}(g)}{e\cdot q^{deg(g)}}\delta m\Bigg[-m\Delta W(\bar{e})W(g)q^{\frac{n}{2}+m}
				+\{(q^n-m)+mq^{\frac{n}{2}+m}\}\Bigg],
			\end{aligned}
		\end{equation}
		where $\Delta=\frac{mr+s-1}{\delta}+2$. 
		Then using the equation (\ref{eq1}), it follows that given $\delta>0$, 
		$\Psi_{\bar{c}}(\overline{L_{{q^n-1}}},x^n-1)>0$ if we have
		\begin{equation}
			q^{\frac{n}{2}-m}>m{W(\bar{e})}W(g)\Delta.
		\end{equation}
	\end{proof}
	\begin{cor}
		Let $q,n\in\mathbb{N}$ such that $q$ is a  prime power, $n>4$ and	let $\beta\in\Field_{q^{n}}^*$. Then for any $c_{1},c_{2}\in\Field_{q}^*$, there exists an element $\alpha\in\Field_{q^{n}}$ such that both $\alpha$, $\alpha+\beta$ are primitive and at least one of them is normal together with $N_{\Field_{q^n}/\Field_q}(\alpha)=c_{1}$ and $N_{\Field_{q^n}/\Field_q}(\alpha+\beta)=c_{2}$ if we have
		\begin{equation}\lb{eq11}
			q^{\frac{n}{2}-2}\geq 2W({L_{{q^n-1}}})^{2}W(x^{n}-1).
		\end{equation}	
	\end{cor}
	\section{Working example}\lb{Sec5}
	In this part, we will identify those $(q,n)$ that might not belong to $\mathcal{S}_{2}$ for $p=3$, $q=3^k$. For $n=5$, we have to go through an exhaustive computation and thus a different approach is required in this case. In this article, we consider the cases $n\geq 6$.
	Further, we divide our computations into two parts. In the first part, we determine the exceptions $(q,n)$ when $n\geq 9$, and in the next part, we carry out computations for $n=6, 7, ~\text{and}~ 8$. We use SageMath \cite{Sm}  as the computational tool for all nontrivial computations needed in this article.
	\subsection{\textbf{Part I}}
	In this part, we assume that $n\geq 9$. Then the equation \ref{eq11} and Lemmas \ref{L3}, \ref{L4} together implies that $(q,n)\in\mathcal{S}_{2}$ if
	\begin{equation}\lb{eq12}
		\begin{aligned}
			q^{\frac{n}{2}-2}&>2~\mathcal{C}^{2}q^{\frac{2n}{\nu}}2^{n}.
		\end{aligned}
	\end{equation}
	Moreover, from the proposition (\ref{Prop4.4}), $\Psi_{\bar{c}}(\overline{L_{{q^n-1}}},x^n-1)>0$ if
	\begin{equation}\lb{eq13}
		\begin{aligned}
			q^{\frac{n}{2}-2}>2{W(e)}^2W(g)\Delta.
		\end{aligned}
	\end{equation}
	Choosing $\nu=11.2$, the inequality ($\ref{eq12}$) is satisfied for $ k\geq 291$ and for $n\geq 7$. Further, if $3\leq k\leq 290$, Table \ref{Table1} gives the values of $n_{k}$, such that the inequality (\ref{eq12}) satisfies for $n\geq n_{k}$ and appropriate choices of $\nu$. 
	\begin{lem}
		Let $q$ and $n$ be positive integers where $q = 3^k$ with $k\geq 3$ and $n\geq 9$. Then we have  $(q,n)\in\mathcal{S}_{2}$.
	\end{lem}
	\begin{proof}
		From the Table \ref{Table1}, it is clear that $(q,n)\in\mathcal{S}_{2}$ for all $k\geq 61$ and $n\geq 9$. For the remaining pairs, we first verify the inequality $q^{\frac{n}{2}-2}\geq 2\mathcal{C}^{2}q^{\frac{2n}{\nu}}W(x^{n}-1)$ and then, the pairs do not satisfy this, we test the inequality (\ref{eq11}). Consequently, we get $(q,n)\in\mathcal{S}_{2}$ unless $(3^3,10),(3^3,13),(3^4,9)$ and $ (3^4,10)$. For these pairs, choose $e,g$ as given in the Table \ref{Table3} and after verifying (\ref{eq13}), we find that all of the pairs $(q,n)\in\mathcal{S}_{2}$.
	\end{proof}
	\begin{center}
		\begin{table}[h]
			\centering
			\caption{Values of $n_{k}$ such that $3\leq k\leq 290$}
			\begin{tabular}{ccc}
				\hline $\nu$ & $k$ & $n_k$ \\
				\hline 9.0 & \{3\} & 151 \\
				8.5 & \{4\} & 65\\
				8 & \{5\} & 42 \\
				8 & \{6\} & 33 \\
				8 & \{7\} & 27 \\
				8 & \{8\} & 24 \\
				8 & \{9\} & 21 \\
				7.5 & \{10\} & 20 \\
				7.9 & \{11,12\} & 18 \\
				7.5 & \{13\} & 17 \\
				7.5 & \{14\} & 16 \\
				7.9 & \{15,16\} & 15 \\
				7.9 & \{17,18,19\} & 14 \\
				7.9 & \{20,21,22,23\} & 13 \\
				7.9 & \{24,25,\ldots,29\} & 12 \\
				8.7 & \{30,31,\ldots,39\} & 11 \\
				8.7 & \{40,41,\ldots,60\} & 10 \\
				9.5 & \{61,62,\ldots,109\} & 9 \\
				10.2 & \{110,111,\ldots,290\} & 8 \\
				\hline
			\end{tabular}
			\label{Table1}
		\end{table}
	\end{center}
	Next we consider the cases $k=1,2$. Let us express $n=n'\cdot 3^{i}$; $i\geq 0$, where $gcd(n{'},3)=1$. As a consequence, we have $W(x^n-1)=W(x^{n'}-1)$. We now divide our discussion in two parts:
	\begin{itemize}
		\item $n'|q^{2}-1$
		\item $n'\nmid q^{2}-1$.
	\end{itemize} 
	Initially, we suppose that $n{'}|q^{2}-1$ and establish the following lemma:
	\begin{lem}
		Let $q=3^k; k=1,2, n\geq 9$ and $n{'}|q^{2}-1$. Then $(q,n)\in\mathcal{S}_{2}$ with possible exceptions $(3,9)$ and $(3,12)$.
	\end{lem}
	\begin{proof}
		Rewrite the inequality (\ref{eq12}) as follows
		\begin{equation}\lb{eq14}
			\begin{aligned}
				q^{{n'\cdot 3^{i}}/{2}-2}&>2~\mathcal{C}^{2}~q^{{2n'\cdot 3^{i}}/{\nu}}~2^{n'}
			\end{aligned}
		\end{equation}\\
		\textbf{Case $(i)$}: Let $k=1$. Taking $\nu=7$, the inequality (\ref{eq14}) is satisfied for $i\geq 3$, when $n'=4,8$ and for $i\geq 4$, when $n'=1,2$. Thus for $n\geq 9$, $(3,n)\in\mathcal{S}_{2}$ unless $n=9,12,18,24,27,36,54,72$. While testing the inequality (\ref{eq11}) with these values, we find that $(3,n) \in\mathcal{S}_{2}$ unless $n=9, 12, 18, 24$. For these specific values of $n$, choose $e$, $g$ as provided in Table \ref{Table3} to satisfy the inequality (\ref{eq13}), resulting in $(q,n)\in\mathcal{S}_{2}$ except when $n=9,12$.\\
		\textbf{Case $(ii)$}: Let $k=2$. Taking $\nu=7$, the inequality (\ref{eq11}) is satisfied for $i\geq 1$, when $n'=40, 80$; for $i\geq 2$, when $n'=5,8,10,16,20$; for $i\geq 3$, when $n'=2,4$ and $i\geq 4$, when $n{'}=1$. Thus for $n\geq 9$, $(9,n)\in\mathcal{S}_{2}$ unless $n=9,10,12,15,16,18,20,24,27,30,36,40,48,60,80$. While testing the inequality (\ref{eq11}) with these values, we find that $(9,n) \in\mathcal{S}_{2}$ except when $n=9,10,12,15,16$.  For these specific values of $n$, choose $l, g$ as provided in Table \ref{Table3} to satisfy the inequality (\ref{eq14}), resulting in $(q,n)\in\mathcal{S}_{2}$ for all $n$.
	\end{proof}
	Let $n$ be represented as $n = n'\cdot q^i$; $i\geq 0$, where $q$ is a prime power such that $gcd(n^\prime,q)=1$. Additionally, let $d$ denote the order of $q$ modulo $n'$, with $q$ and $n'$ being co-prime. According to [{\cite{RH}}, Theorems $2.45$ and $2.47$], $x^{n'}-1$ can be expressed as the product of irreducible polynomials in $\Field_{q}[x]$, each having a degree less than or equal to $d$.
	
	Let $N_{0}$ denote the number of irreducible factors of $x^{n'}-1$ over $\Field_{q}$, each having a degree less than $d$, and let $\rho(q,n')$ represent the ratio $\frac{N_{0}}{n'}$. Note that the number of irreducible factors of $x^{n}-1$ over $\Field_{q}$ is identical to the number of irreducible factors of $x^{n'}-1$ over $\Field_{q}$, which implies $n\rho(q,n)=n'\rho(q,n')$.
	
	In fact, when $p=3$ and $n=n{'}\cdot 3^{i}$, where $3\nmid n{'}$, the subsequent lemma provides bounds for $\rho(q,n)$, which will be used in further discourse.
	\begin{lem}\textbf{({\cite{SS}}, Lemmas 6.1, 7.1)}\lb{lem6.3}
		Let $q=3^k$ and $n{'}>4$ be such that $3\nmid n{'}$. Then the following hold:
		\begin{itemize}
			\item[(i)] If $n{'}=2\cdot gcd(q-1,n{'})$, then $d=2$ and $\rho(q,n{'})=1/2$.\\
			\item[(ii)]If $n{'}=4\cdot gcd(q-1,n{'})$ and $q\equiv 1(mod~ 4)$, then $d=4$ and  $\rho(q,n{'})=3/8$.\\
			\item[(iii)] $\rho(3,16)=5/16$, otherwise $\rho(3,n{'})\leq 1/4$.\\
			\item[(iv)] Elsewhere, $\rho(q,n{'})\leq 1/3$.
		\end{itemize}
	\end{lem}
	\begin{lem}\lb{lem6.4}
		Let $q=p^k$; $k\in\mathbb{N}$ and $n=n'\cdot q^{i}$; $i\geq 0$, where $n'$ and $q$ are co-prime with $n'\nmid q-1$. Suppose $d(>2)$ be the order of $q$ modulo $n'$. Moreover, let $e$ be equal to $L_{q^n-1}$ and $g$ be the product of all irreducible factors of $x^{n'}-1$ having degrees less than $d$. Then, in accordance with the proposition (\ref{Prop4.4}), we have $\Delta <2n'$. 
	\end{lem}
	\begin{proof}
		We omit the proof of the lemma, as it follows from [\cite{MASI}, Lemma 10].
	\end{proof}
	\begin{lem}
		Let $q=3^k; k=1,2, n\geq 9$ and $n{'}\nmid q^{2}-1$. Then $(q,n)\in\mathcal{S}_{2}$ with possible exceptions  $(3,10)$ and $(3,16)$.
	\end{lem}
	\begin{proof}
		\textbf{Case $(i)$}: Let $k=1$. It is clear that $n'\geq 3$ and from Lemma \ref{lem6.3}, it follows that $\rho(3,n')\leq 1/4$ except $n'=16$. We consider $e$ as $L_{q^n-1}$ and $g$ as the product of all irreducible factors of $x^{n'}-1$ with degrees less than $d$. Then by Lemma \ref{lem6.4} and the inequality (\ref{eq14}), the pairs $(q,n)\in\mathcal{S}_{2}$, if the inequality $3^{\frac{n}{2}-2}>2~\mathcal{C}^{2}3^{\frac{2n}{\nu}}2^{\frac{n}{4}}2n$ is true. This  inequality is valid for $n\geq 277$ and $\nu=9.2$. For $ n\leq 276$, we verify the inequality (\ref{eq11}) and find that $(3,n)\in\mathcal{S}_{2}$ except for $n=10, 11, 13, 14, 15, 20, 22$. For these exceptions, we choose the values of $e$ and $g$ as mentioned in Table \ref{Table3}, and test the inequality \ref{eq13}. Hence we have, $(3,n) \in\mathcal{S}_{2}$ except for $n=10$.
		
		Now we consider the case $n'=16$. Here we test the inequality $3^{\frac{16\cdot 3^{i}}{2}-2}>64~\mathcal{C}^{2}~3^{\frac{32\cdot 3^{i}}{\nu}}~2^{5/16}$ and get that $(3,16\cdot 3^{i})\in\mathcal{S}_{2}$ for all $i\geq 2$ and $\nu=9.2$. Among the exceptionas $(3,16),(3,48)$, the last one satisfies the inequality (\ref{eq11}), however for the remaining one , the inequality (\ref{eq14}) does not hold for any choice of $e$ and $g$. 
		
		\textbf{Case $(ii)$}: Let $k=2$. If $n^{'}\nmid q^{2}-1$, then $n'\nmid q-1$ and by Lemma \ref{L4} $W(x^{n'}-1)\leq 2^{3n'/4}$. Then the inequality (\ref{eq11}) is true if $9^{{n}/{2}-2}>2~ \mathcal{C}^{2}9^{{2n}/{\nu}}2^{{3n}/{4}}$ is true and the latter is true for $n\geq 335$ and $\nu=9.2$. For $9\leq n\leq 334$, we verify the inequality (\ref{eq11}) and find that $(9, n)\in\mathcal{S}_{2}$ except for $n = 11, 14$. For these exceptional cases, we choose values for $e$ and $g$ as listed in Table \ref{Table3} for which the inequality (\ref{eq14}) is true and get that $(9, n) \in \mathcal{S}_{2}$.
	\end{proof}
	
	\subsection{\textbf{Part II}}
	In this section, we shall verify the cases where $n$ takes values of $6, 7,~\text{or}~ 8$. We shall also utilize the following result for further calculations.
	\begin{lem}\textbf{({\cite{WQA}}, Lemma 5.4)}\lb{L5.1}
		For any $M\in\mathbb{N}$ satisfying $\om(M)\geq 2828$, we have $W(M)<M^{\frac{1}{13}}$.
	\end{lem}First, let us suppose that $\om(q^n-1)\geq 2828$. Clearly, in this part we have $W(x^n-1)\leq 2^8$. Then (\ref{eq11}) and the Lemma \ref{L5.1} together implies that $(q,n)\in \mathcal{S}_{2}$ if we have $q^{\frac{n}{2}-2}>2^{9}\cdot q^{\frac{2n}{13}}$ or equivalently if $q^{\frac{9n}{26}-2}>2^{9}$, that is if we have $q^n>512^{\frac{26n}{9n-52}}$. This holds for $q^n>512^{78}$, which is true for $\om(q^n-1)\geq 2828$. For further progress we follow \cite{WQA}. We now assume that $88\leq \om(q^n-1)\leq 2827$. Choosing $g={x^n-1}$ and $e$ to be the product of least $88$ primes dividing $L_{q^n-1}$ i.e., $W(e)=2^{88}$. Then $r \leq 2739$ and $\delta$
	will be least when $\{p_{1},p_{2},\ldots,p_{2739}\} = \{461,463,\ldots,25667\}$. This gives $\delta>0.0044306$ and $\Delta<1.2362\times 10^6$, hence $2\Delta W(e)^{2} W(g)<6.07\times 10^{61}=T$(say). By Proposition \ref{Prop4.4}, $(q,n)\in\mathcal{S}_{2}$ if we have $q^{\frac{n}{2}-2}>T$ or $q^n>T^{\frac{2n}{n-4}}$. But $n\geq 6$ implies $\frac{2n}{n-4}\leq 6$. Therefore,
	whenever $q^{n}>T^{6}$ or $q^{n}>5.00187\times 10^{370}$ then $(q,n)\in\mathcal{S}_{2}$. Hence, $\om(q^n-1)\geq 154$ gives $(q,n)\in\mathcal{S}_{2}$. Iterating the process (of Proposition \ref{Prop4.4}) for the values in Table \ref{Table2}, we get that $(q,n)\in\mathcal{S}_{2}$, if we have $q^{\frac{n}{2}-2}>2.7829\times 10^{10}$. This determines  the cases $n=6, q>2.7829\times 10^{10}$; $n=7, q>(2.7829\times 10^{10})^{{2}/{3}} $ and $n=8, q>(2.7829\times 10^{10})^{1/2}$. Thus, the only possible exceptions are $(3,6),(3^2,6),\ldots,(3^{21},6)$; $(3,7),(3^2,7),\ldots,(3^{14},7)$ and $(3,8),(3^2,8),\ldots,(3^{10},8)$. For these pairs,  first we verify the inequality (\ref{eq11}) and then for the remaining pairs, choose $e,g$ as mentioned in the Table \ref{Table4} and verify (\ref{eq13}). Hence, only possible exceptions are $(3,6),(3^2,6),(3^3,6),(3,7),(3,8)~\text{and}~(3^2,8)$.
	\begin{center}
		\begin{table}[h]
			\centering
			\caption{Values of $W(e)$, $\delta$ and $\Delta$ such that Proposition \ref{Prop4.4} holds}
			\begin{tabular}{cccccc}
				\hline $Sr.no$ & $c\leq \om(q^n-1)\leq d$ & $W(e)$ & $\delta>$& $\Delta<$& $2\Delta W(e)^{2}W(g)<$\\
				\hline $1$ & $a=17, b=153$ & $2^{17}$ & $0.0304555$ &$8900.2013271$&$7.8287\times 10^{16}$ \\
				$2$ & $a=9, b=54$ & $2^{9}$ &$0.03063362$& $2907.3044905$& $3.9022\times 10^{11}$\\
				$3$ & $a=8, b=40$ & $2^{8}$ &$0.0761466$& $829.3503968$& $2.7829\times 10^{10}$\\
				\hline
			\end{tabular}
			\label{Table2}
		\end{table}
	\end{center}
	The above discussions leads us to conclude the following.
	\begin{thm}
		Let $q,k,n\in\mathbb{N}$ such that $q=3^k$ and $n\geq 6$. The $(q,n)\in\mathcal{S}_{2}$ unless the following possible exceptions:
		\begin{itemize}
			\item[1.] $q=3,3^2,3^3~\text{and}~ n=6$;
			\item[2.] $q=3,3^2~\text{and}~ n=8$;
			\item[3.] $q=3~\text{and}~ n=7,9,10,12,16$.
		\end{itemize}
	\end{thm}

	\begin{center}
		\begin{table}[h]
			\centering
			\caption{Pairs $(q,n)\in\mathcal{S}_{2}(6\leq n\leq 8)$ by Sieve Technique}
			\begin{tabular}{ccccccc}
				\hline $Sr.no$ & $(q,n)$ & $e$  & $g$&$\delta$&$\Delta$&\\
				\hline 
				$1$&$(3^4,6)$&$1$&$1$&$0.459261859983295$&$25.9514772691991$&\\
				$2$&$(3^5,6)$&$1$&$x+2$&$0.447087519356890$&$31.0770809677259$&\\
				$3$&$(3^6,6)$&$1$&$x+2$&$0.407896267111061$&$33.8708481743972$&\\
				$4$&$(3^7,6)$&$1$&$x+2$&$0.508470251621857$&$27.5668841166895$&\\
				$5$&$(3^8,6)$&$1$&$x+2$&$0.377733579667098$&$52.3000024957934$&\\
				$6$&$(3^9,6)$&$1$&$x+2$&$0.531452453574178$&$33.9878097949682$&\\
				$7$&$(3^{10},6)$&$1$&$x+2$&$0.0589802048533907$&$324.141980470041$&\\
				$8$&$(3^{11},6)$&$1$&$x+2$&$0.527461761956088$&$30.4380804105545$&\\
				$9$&$(3^{12},6)$&$1$&$x+2$&$0.788944166092470$&$21.0127522892943$&\\
				$10$&$(3^{13},6)$&$1$&$x+2$&$0.514520279376933$&$42.8147177899972$&\\
				$11$&$(3^{14},6)$&$1$&$x+2$&$0.0158565410261152$&$1452.50550193259$&\\
				$12$&$(3^{15},6)$&$1$&$x+2$&$0.435349772783422$&$59.4250902674457$&\\
				$13$&$(3^{16},6)$&$1$&$x+1$&$0.506021983577781$&$43.5001732761123$&\\
				$14$&$(3^{17},6)$&$1$&$x+1$&$0.529063818388427$&$41.6927540121110$&\\
				$15$&$(3^{18},6)$&$1$&$x+1$&$0.549184527812597$&$36.5967503412324$&\\
				
				$16$&$(3^{19},6)$&$1$&$x+1$&$0.550290107689489 $&$40.1616890919101$&\\
				$17$&$(3^{20},6)$&$1$&$x+1$&$0.403312899984737$&$68.9455403013932$&\\
				
				$18$&$(3^{21},6)$&$1$&$x+1$&$0.482470736843123$&$53.8166141299650$&\\
				\hline
				$19$&$(3^2,7)$&$1$&$1$&$0.880659271367610$&$8.81307765111286$&\\
				$20$&$(3^3,7)$&$1$&$x+2$&$0.994049514027501$&$8.03591663728131$&\\
				$21$&$(3^4,7)$&$1$&$x+2$&$0.925423322543100$&$11.7252789947714$&\\
				$22$&$(3^5,7)$&$1$&$x+2$&$0.970001158998250 $&$8.18555961953322$&\\
				$23$&$(3^6,7)$&$1$&$x+2$&$0.938884907082541 $&$17.9763991164907$&\\
				$24$&$(3^7,7)$&$1$&$x+2$&$0.995131839337739 $&$12.0489197558537$&\\
				$25$&$(3^8,7)$&$1$&$x+2$&$0.925419354656423 $&$16.0476854461580$&\\
				$26$&$(3^9,7)$&$1$&$x+2$&$0.998164732621425 $&$10.0147091342228$&\\
				$27$&$(3^{10},7)$&$1$&$x+2$&$0.966344851863111 $&$13.3830999138579$&\\
				$28$&$(3^{11},7)$&$1$&$x+2$&$0.998170134660812 $&$10.0146657590777$&\\
				$29$&$(3^{12},7)$&$1$&$x+2$&$0.877317235655992 $&$28.2162864984664$&\\
				$30$&$(3^{13},7)$&$1$&$x+2$&$0.998170173334930 $&$8.01099908641203$&\\
				$31$&$(3^{14},7)$&$1$&$x+2$&$0.995131836990084$&$17.0733796693407$&\\
				\hline
				$32$&$(3^3,8)$&$1$&$x+2$&$0.196647108479655 $&$68.1082692774249$&\\
				$33$&$(3^4,8)$&$1$&$x+2$&$0.736789913062364$&$21.0013459085114$&\\
				$34$&$(3^5,8)$&$1$&$x+2$&$0.512573068211997$&$27.3622377105137$&\\
				$35$&$(3^6,8)$&$1$&$x+2$&$0.359216168215832$&$68.8121374358067$&\\
				
				$36$&$(3^7,8)$&$1$&$x+1$&$0.478010818428838$&$37.5640486461727$&\\
				
				$37$&$(3^8,8)$&$1$&$x+1$&$0.870923243421729$&$18.0748953547227$&\\
				
				$38$&$(3^9,8)$&$1$&$x+1$&$0.0784275783761321$&$269.762953221452$&\\
				
				$39$&$(3^{10},8)$&$1$&$x+1$&$0.421258807060023 $&$54.2244274334312$&\\
				\hline
			\end{tabular}
			\label{Table4}
		\end{table}
	\end{center}
	\begin{center}
		\begin{table}[h]
			\centering
			\caption{Pairs $(q,n)\in\mathcal{S}_{2}(n\geq 9)$ by Sieve Technique}
			\begin{tabular}{ccccccc}
				\hline $Sr.no$ & $(q,n)$ & $e$  & $g$&$\delta$&$\Delta$&\\
				\hline 
				$1$&$(3^2,9)$&$1$&$x+2$&$0.287369229453635$&$36.7984369064588$&\\
				$2$&$(3^4,9)$&$1$&$x+2$&$0.358733632763241$&$41.0261707333134$&\\
				\hline
				$3$&$(3^2,10)$&$1$&$x+2$&$0.223207625673814$&$55.7616040839764$&\\
				$4$&$(3^3,10)$&$1$&$x+1$&$0.390305143669876$&$37.8693709961489$&\\
				$5$&$(3^4,10)$&$1$&$x+1$&$0.623809806883578$&$30.8549487381806$&\\
				\hline
				$6$&$(3,11)$&$1$&$1$&$0.570960346628126
				$&$12.5086106862477$&\\
				$7$&$(3^2,11)$&$1$&$x+2$&$0.879613797583913$&$12.2317631041269$&\\
				\hline
				$8$&$(3^2,12)$&$5$&$x+2$&$0.0395087735635420$&$331.040838969402$&\\
				\hline
				$9$&$(3,13)$&$1$&$x+2$&$0.851849342948381
				$&$7.86958250468825$&\\
				$10$&$(3^3,13)$&$1$&$x+2$&$0.548585326900787$&$36.6345391834299$&\\
				\hline
				$11$&$(3,14)$&$1$&$x^2+2$&$0.991770382478721
				$&$7.04148953057416$&\\
				$12$&$(3^2,14)$&$1$&$x+2$&$0.408829006334904$&$36.2441455549059$&\\
				\hline
				$13$&$(3,15)$&$1$&$x+2$&$0.318218151661110
				$&$23.9974880862696$&\\
				$14$&$(3^2,15)$&$1$&$x+2$&$0.137697321077803$&$118.196886582561$&\\
				\hline
				
				$15$&$(3^2,16)$&$5$&$x^2+2$&$0.107160283432735$&$160.640864464221$&\\
				\hline
				
				$16$&$(3,18)$&$1$&$x+1$&$0.0651470072314125$&$155.498992892773$&\\
				\hline
				$17$&$(3,20)$&$5$&$x^4+2$&$0.734318736784925$&$14.2562581467072$&\\
				\hline
				$18$&$(3,22)$&$23$&$x^2+2$&$0.950143284148867$&$11.4722555536054$&\\
				\hline
				$19$&$(3,24)$&$1435$&$x+1$&$0.151781324860484$&$61.2958323975149$&\\
				\hline
			\end{tabular}
			\label{Table3}
		\end{table}
	\end{center}
\end{document}